\documentclass[11pt]{amsart}

\usepackage[colorlinks=true, pdfstartview=FitV, linkcolor=blue, citecolor=blue, urlcolor=blue]{hyperref}

\usepackage[letterpaper,margin=1in]{geometry}
\usepackage{enumerate, amsmath, amsfonts, amssymb, amsthm, mathtools, thmtools, wasysym, graphics,
xcolor, frcursive,xparse,comment,bbm}
\usepackage{graphicx, xspace, ifthen, rotating, pict2e}
\usepackage{multirow,multicol,mathtools}
\usepackage{tikz}
\usepackage{ytableau}
\usepackage[cmtip,all]{xy}
\usepackage{young}
\usepackage{subcaption}

\definecolor{darkblue}{rgb}{0.0,0,0.7} 
\definecolor{darkred}{rgb}{0.7,0,0} 
\definecolor{darkgreen}{rgb}{0, .6, 0} 

\newtheorem{theorem}{Theorem}[section]
\newtheorem{proposition}[theorem]{Proposition}
\newtheorem{lemma}[theorem]{Lemma}
\newtheorem{corollary}[theorem]{Corollary}
\newtheorem{conjecture}[theorem]{Conjecture}

\theoremstyle{definition}

\newtheorem{definition}[theorem]{Definition}
\newtheorem{remark}[theorem]{Remark}
\newtheorem{example}[theorem]{Example}

\numberwithin{equation}{section}

\newcommand{\onep}{1'}
\newcommand{\twop}{2'}
\newcommand{\threep}{3'}

\newcommand{\oneb}{$-$1}
\newcommand{\twob}{$-$2}
\newcommand{\threeb}{$-$3}
\newcommand{\fourb}{$-$4}

\newcommand{\sevb}{$-$7}
\newcommand{\si}{s_i}
\newcommand{\twor}{\color{darkred}{2}}

\newcommand{\twobl}{\color{darkblue}{2}}
\newcommand{\threepg}{\color{darkgreen}{3'}}

\usepackage{graphics, graphicx, xcolor}
\usepackage[vcentermath, enableskew]{youngtab}
\usepackage{lipsum,  verbatim}
\usepackage{mathtools}
\usepackage{txfonts}
\usepackage{multicol}
\newcommand{\defn}[1]{{\color{darkred}\emph{#1}}}
\title{Crystal Structures for Double Stanley Symmetric Functions}

\author[G.~Hawkes]{Graham Hawkes}
\address[G.~Hawkes]{Department of Mathematics, University of California, Davis, One Shields Avenue, Davis, CA 95616, USA}
\curraddr[G.~Hawkes]{Max-Planck-Institut f\"ur Mathematik, Vivatsgasse 7, 53111 Bonn, Germany}
\email{hawkes@math.ucdavis.edu}
\urladdr{https://www.math.ucdavis.edu/~hawkes/}

\thanks{The author acknowledges support from the Max-Planck-Institut f\"ur Mathematik. }


\begin{document}

\begin{abstract}
We relate the combinatorial definitions of the type $A_n$ and type $C_n$ Stanley symmetric functions, via a combinatorially defined ``double Stanley symmetric function," which gives the type $A$ case at $(\mathbf{x},\mathbf{0})$ and gives the type $C$ case at $(\mathbf{x},\mathbf{x})$.  We induce a  type $A$ bicrystal structure on the underlying combinatorial objects of this function which has previously been done in the type $A$ and type $C$ cases.  Next we prove a few statements about the algebraic relationship of these three Stanley symmetric functions.
We conclude with some conjectures about what happens when we generalize our constructions to type $C$.

\end{abstract}

\maketitle

\section{Introduction and Notation}

In this paper we will relate the combinatorial definitions of the type $A_n$  (\cite{Stanley.1984}) and type $C_{n+1}$ \cite{Billey.Haiman.1995}, \cite{Fomin.Kirillov.1996} Stanley symmetric functions.  To do this, we define combinatorially a ``double Stanley symmetric function" and show that it is indeed a symmetric function in two sets of infinite variables.  Precisely, our double Stanley symmetric function gives the type $A$ Stanley symmetric function at $(\mathbf{x},\mathbf{0})$ and gives the type $C$ Stanley symmetric function at $(\mathbf{x},\mathbf{x})$.

 Both the type $A$ and type $C$ functions are Schur positive, and a crystal-theoretic interpretation of these facts has been given in \cite{Morse.Schilling.2016} and \cite{Hawkes.Paramonov.Schilling.2017} respectively. Furthermore, crystal analysis for the (type A) stable limit of double Schubert polynomials is carried out \cite{Lenart.2004} by considering a crystal structure on the underlying combinatorial objects of rc graphs (in other literature known as pipe dreams). In the paper we will carry out this procedure for our new double Stanley symmetric function by considering a crystal sructure for the underlying objects of reduced signed increasing factorizations.   To do this we first write the double Stanley symmetric functions as a sum of characters of certain tableaux (section 2). Next, we introduce explicit crystal operators on these tableaux, which allows us to write the  double Stanley symmetric function as sum of products of Schur functions (section 3). Then, we introduce a notion of conversion which helps us explore the algebraic relationship of all three of these Stanley symmetric functions, in particular, recovering some results of Lam \cite{Lam.1995} (section 4). Although sections 2-4 are restricted to the type $A$ case, we conclude (section 5) with a quick survey of some results (without proof) and conjectures about the type $C$ case.\\

Throughout the paper, when some $k \in \mathbb{N}$ is specified $x$ will refer to the list of variables $(x_1, \ldots, x_k)$ and $y$ will refer to the list of variables $(y_1, \ldots, y_k)$.  On the other hand  $\mathbf{x}$ will refer to the infinite list of variables $(x_1, x_2, \ldots)$ and $\mathbf{y}$ will refer to the infinite list of variables $(y_1, y_2,\ldots)$.  If the polynomial $P(x)$ or $P(x,y)$ is defined for arbitrary $k$ then $P(\mathbf{x})$ or, respectively, $P(\mathbf{x},\mathbf{y})$ will represent the corresponding function obtained by letting $k \rightarrow \infty$.\\

The $A_n$ Coxeter system is defined as the Coxeter system with generators, $s_1, \ldots, s_n$ and relations $(s_is_j)^{m_{ij}}=1$ where $m_{ij}$ is an integer determined as follows:

\begin{itemize}
\item{If $|i-j|=0$,  $m_{ij}=1$}.
\item{If $|i-j|=1$, $m_{ij}=3$}.
\item{If $|i-j|>1$, $m_{ij}=2$}.
\end{itemize}

By abuse of notation, we will also refer to the corresponding Coxeter group of size $(n+1)!$ as $A_n$.  The $C_{n+1}$ Coxeter system is defined as the Coxeter system with generators, $s_0, s_1, \ldots, s_n$ and relations $(s_is_j)^{m_{ij}}=1$ where $m_{ij}$ is an integer determined as follows:

\begin{itemize}
\item{If $|i-j|=0$,  $m_{ij}=1$}.
\item{If $i>0$ and $j>0$, and $|i-j|=1$, $m_{ij}=3$}.
\item{If $i=0$ or $j=0,$ and $|i-j|=1$, $m_{ij}=4$}.
\item{If $|i-j|>1$, $m_{ij}=2$}.
\end{itemize}

Similarly, we will sometimes refer to the corresponding group of size $2^{n+1} (n+1)!$ itself as $C_{n+1}$.  Given the relations above one can define two types of symmetric functions, indexed, respectively, by elements of $A_n$ and $C_{n+1}$.\\

First, suppose $\omega \in A_n$.  A \emph{reduced word} for $\omega$ is an expression, $u$, for $\omega$ using the generators  $s_1, \ldots, s_n$ such that no other such expression for $\omega$ is shorter than $u$.  Given a fixed $k$, a \defn{reduced increasing factorization} ($RIF_k$), $v$, for $\omega$ is a reduced word $u$, for $\omega$ along with a subdivision of $u$ into $k$ parts such that each part is increasing under the order $s_1 < \cdots <s_n$.   The \emph{weight} of $v$ is the vector whose $i^{th}$ entry records the number of generators in the $i^{th}$ subdivision of $v$. The type $A$ Stanley symmetric polynomial \cite{Stanley.1984} in $k$ variables for $\omega$ is:

\begin{eqnarray*}
F_{\omega}^A(x)=\sum_{v \in RIF_k({\omega})} x ^ {wt(v)},
\end{eqnarray*}
where $RIF_k({\omega})$ is the set of reduced increasing factorizations of $\omega$, and $wt(v)$ is the weight of $v$.  Letting $k \rightarrow \infty$ in the  type $A$ Stanley symmetric polynomial gives the type $A$ Stanley symmetric function for $\omega$. \\

Now suppose $\omega \in C_{n+1}$.  A \emph{reduced word} for $\omega$ is an expression, $u$, for $\omega$ using the generators  $s_0, s_1, \ldots, s_n$ such that no other such expression for $\omega$ is shorter than $u$.  Given a fixed $k$, a \defn{reduced unimodal factorization} $(RUF_k)$, $v$, for $\omega$ is a reduced word $u$, for $\omega$ along with a subdivision of $u$ into $k$ parts such that each part is unimodal (i.e., decreasing and then increasing) under the order $s_0<s_1 < \cdots <s_n$.  The \emph{weight} of $v$ is the vector whose $i^{th}$ entry records the number of generators in the $i^{th}$ subdivision of $v$. The type $C$ Stanley symmetric polynomial \cite{Billey.Haiman.1995}, \cite{Fomin.Kirillov.1996}, in $k$ variables for $\omega$ is:

\begin{eqnarray*}
F_{\omega}^C(x)=\sum_{v \in RUF_k({\omega})} 2^{ne(v)} x ^ {wt(v)},
\end{eqnarray*}
where $ne(v)$ is the number of nonempty subdivisions of $v$, $RUF_k({\omega})$ is the set of reduced unimodal factorizations of $\omega$, and $wt(v)$ is the weight of $v$. Letting $k \rightarrow \infty$ in the  type $C$ Stanley symmetric polynomial gives the type $C$ Stanley symmetric function for $\omega$. \\

Next consider the generators $s_{-n},\ldots s_{-1},s_0, s_{1},\ldots, s_{n}$ and impose relations $(s_is_j)^{m_{ij}}=1$ where $m_{ij}$ is an integer determined as follows:

\begin{itemize}
\item{If $|(|i|-|j|)|=0$,  $m_{ij}=1$}.
\item{If $i \neq 0$ and $j \neq 0$, and $|(|i|-|j|)|=1$, $m_{ij}=3$}.
\item{If $|(|i|-|j|)|>1$, $m_{ij}=2$}.
\item{If $i=0$ or $j=0,$ and $|(|i|-|j|)|=1$, $m_{ij}=4$}.
\end{itemize}
Of course, the resulting system is not Coxeter, for instance, the relations imply that $s_{-i}=s_i$ holds,   \footnote{This makes sense on the level of Weyl groups: the reflection over the plane perpendicular to the $i^{th}$ simple root is equal to the reflection over the plane perpendicular to the opposite of the $i^{th}$ simple root.} so the generating set is obviously not minimal.   \\

In this setting, a \emph{reduced word} for $\omega$ is an expression, $u$, for $\omega$ using the  generators $s_{-n},\ldots, s_{-1}, s_0, s_1, \ldots, s_n$ such that no other such expression for $\omega$ is shorter than $u$.  Given a fixed $k$, a \defn{reduced signed increasing factorization} ($RSIF_k)$, $v$, for $\omega$ is a reduced word $u$, for $\omega$ along with a subdivision of $u$ into $k$ parts such that each part is increasing under the order $s_{-n}<\cdots s_{-1} < s_0 < s_1 < \cdots <s_n$.\\

The \emph{double weight} of $v$, denoted $(dw(v,1),dw(v,2))$ is the pair $(X,Y)$, where the $i^{th}$ entry of $X$ records the number of generators with negative index in the $i^{th}$ subdivision of $v$, and  the $i^{th}$ entry of $Y$ records the number of generators with nonnegative index in the $i^{th}$ subdivision of $v$.   For instance, $v=(s_{-3}s_{-2}s_{1})(s_{-5}s_2s_3)(s_{-4}s_{-3})$ is an $RSIF$ (with $k=3$) for $\omega=s_3s_2s_1s_2s_3s_5s_4s_3$ with double weight $((2,1,2),(1,2,0))$. We define the \emph{double Stanley symmetric polynomial} in $k$ variables for $\omega \in C_{n+1}$ to be:

\begin{eqnarray*}
F_{\omega}^d(x,y)=\sum_{v \in RSIF_k({\omega})} x ^ {dw(v,1)} y ^ {dw(v,2)},
\end{eqnarray*}
where $RSIF_k({\omega})$ is the set of reduced signed increasing factorizations of $\omega$ into $k$ parts.  Letting $k \rightarrow \infty$ in the  double Stanley symmetric polynomial gives the double Stanley symmetric function for $\omega$. We will frequently use the shorthand $i$ for $s_i$ and $\bar{i}$ for $s_{-i}$ when it is clear we are discussing expressions of Coxeter elements. For instance, $v$ above may be rewritten: $v=(\bar{3}\bar{2}1)(\bar{5}23)(\bar{4}\bar{3})$.

In sections 2-4 we consider the special case where $\omega \in A_n$.  In section 5 we give a short overview of what happens when $\omega \in C_{n+1}$ (in general the double Stanley ``symmetric" function may not be symmetric if $\omega \notin A_n$).  The fact that $F_{\omega}^d(\mathbf{x},\mathbf{y})$ \textit{is} symmetric for $\omega \in A_n$ is not obvious and will require some effort to show. For now, we simply note some equalities that do follow immediately from the constructions: For any $\omega \in A_n \subseteq C_{n+1}$ we can define all three Stanley symmetric functions mentioned, and we have:  $F_{\omega}^d(\mathbf{0},\mathbf{x})= F_{\omega}^A(\mathbf{x})$  and $F_{\omega}^d(\mathbf{x},\mathbf{x})= F_{\omega}^C(\mathbf{x})$  and $F_{\omega}^d(\mathbf{x},\mathbf{y})=F_{\omega^{-1}}^d(\mathbf{y},\mathbf{x})$.

\section{Expansion in Terms of Primed Tableaux}

In this section we expand the function $F_{\omega}^d$ for $\omega \in A_n$ in terms of a certain generating function for primed tableaux. Now, we fix some $k \in \mathbb{N}$ for the remainder of this section. We will work over the alphabet $\bar{X}_k'=\{\bar{k}<\cdots < \bar{2} < \bar{1} < 1'< 1 < 2'< 2 < \cdots < k' < k\}$ (these elements are not related to generators $s_i$ or $s_{-i}$.  This association is only made when numbers appear within parenthesis inside a factorization or within an Edelman-Greene tableau (defined later)). An element in this alphabet is called \defn{marked} if it is barred or it is primed, and called \defn{unmarked} otherwise. The subset of $\bar{X}_k'$ which contains no primed letters is denote $\bar{X}_k$.  The subset of $\bar{X}_k'$ which contains no barred letters is denoted $X_k'$.  Inside tableaux, barred entries will be represented using a small -, for example  $\young(\fourb\threeb\oneb22)$ represents the one row tableau with entries $\{\bar{4}, \bar{3}, \bar{1}, 2, 2\}$ because moving the bar in front of the number makes it easier to see. 

The definition of primed tableau (given below)  appears as a specific case of a more general definition. (This more general definition will be needed later.)

\begin{definition}

 Fix partitions $\mu \subseteq \lambda$. Fix vectors $X$ and $Y$ in $\mathbb{Z}_{\geq 0}^k$. Finally, fix  $0 \leq j \leq k$. We define the set of \defn{primed signed tableaux} corresponding to these parameters, by declaring that  $T \in PST(\lambda/\mu,X,Y,j)$ if:

\begin{enumerate}
\item{$T$ has shape $\lambda/\mu$.}
\item{$T$ has entries from $\bar{X}_k'$.}
\item{The rows and columns of $T$ are weakly increasing.}
\item{Each row of $T$ has at most one marked $i$ and each column has at most one unmarked $i$.}
\item{$T$ contains $Y(i)$ unmarked $i$s.}
\item{$T$ contains $X(i)$ barred $i$s (and no primed $i$s) for each $i > j$.}
\item{$T$ contains $X(i)$ primed $i$s (and no barred $i$s) for each $i \leq j$.}
\end{enumerate}

\end{definition}

\begin{example}
Let $\lambda=(4,3,2,2)$. The following lies in $PST$\Bigg(
$\lambda/\emptyset
$,
$\begin{bmatrix}
1\\
2\\
2\\
2\\
\end{bmatrix}$,
$\begin{bmatrix}
2\\
2\\
0\\
1\\
\end{bmatrix}$,
2 \Bigg):
$\young(\fourb\threeb14,\fourb\onep\twop,\threeb1\twop,22) $  \\

\end{example}

\begin{definition}
Let $X, Y \in \mathbb{Z}_{\geq 0}^k$. A  \defn{primed tableau} of shape $\lambda/\mu$ and double weight $(X,Y)$ is an element of  $PST(\lambda/\mu,X,Y,k)$.
\end{definition}
\begin{definition}
Let $X, Y \in \mathbb{Z}_{\geq 0}^k$. A  \defn{signed tableau} of shape $\lambda/\mu$ and double weight $(X,Y)$ is an element of  $PST(\lambda/\mu,X,Y,0)$.
\end{definition}

In other words, a primed tableau is a primed signed tableau in $X_k'$ and a signed tableau is a primed signed tableau in $\bar{X}_k$.  For shorthand, we also refer to the set of primed tableaux, $PST(\lambda/\mu,X,Y,k)$, as $PT_k(\lambda/\mu)$ with double weight $(X,Y)$.  We now define a polynomial in the variables $(x,y)$ by:

\begin{eqnarray*}
R_{\lambda/\mu}(x,y)=\sum_{T \in PT_k(\lambda/\mu)} x^{dw(T,1)}y^{dw(T,2)},
\end{eqnarray*}
where $(dw(T,1),dw(T,2))$ represents the double weight of the primed tableau, $T$ . Our goal is to show that $F_{\omega}^d$ expands in terms of $R_{\lambda}$.

Let $\omega \in A_n$.  An \defn{Edelman-Greene tableau} for $\omega$ is a tableau in $\{s_1, \ldots, s_n \}$ which is row-wise and column-wise increasing with respect to the order $s_1 < \cdots < s_n$, and which, if read by rows, left to right, bottom to top, forms a reduced word for $\omega$. For viewing convenience we will write $\young(i)$ to mean $\young(\si)$.
We use Edelman-Greene insertion,
\cite{Edelman.Greene.1987}, to create a bijection between $RSIF_k(\omega)$ and pairs of tableaux, $(P,Q)$, where $P$ is an Edelman-Greene tableau for $\omega$, and $Q$ is a primed tableau of the same shape. This bijection is described below.
\begin{definition}
\defn{Primed-Recording Edelman-Greene map}. Suppose $v \in RSIF_k(\omega)$. Create the insertion tableau $P$ by applying Edelman-Greene insertion to $|v|$, the expression obtained by ignoring the subdivisions of $v$ and replacing $s_{-i}$ by $s_i$ for each $i$.  Create the recording tableau, $Q$, as follows: Each time a box  is added to $P$ say in position $(i,j)$ add a box to $Q$ in position $(i,j)$ and fill it as follows: Suppose box $(i,j)$ was added to $P$ when $|v|_r$ was inserted. Let $l$ be the subdivision of $v$ in which $v_r$ occurs in $v$.  If $v_r$ is barred in $v$, fill box $(i,j)$ of $Q$ with $l'$. If $v_r$ is unbarred in $v$, fill box $(i,j)$ of $Q$ with $l$.
\end{definition}

\begin{example}
Let $v=(\bar{3}\bar{2}14)(\bar{3}\bar{2})(\bar{4}13)$. \\
$\{\young(3), \young(\onep)\} \rightarrow
\left(\young(2,3), \, \young(\onep,\onep)\right\} \rightarrow
\left\{\young(1,2,3), \, \young(\onep,\onep,1)\right\}  \rightarrow 
\left\{\young(14,2,3), \, \young(\onep1,\onep,1) \right\}  \rightarrow
\left\{\young(13,24,3), \, \young(\onep1,\onep\twop,1)\right\}  \rightarrow
\left\{\young(12,23,34), \, \young(\onep1,\onep\twop,1\twop)\right\} $\\

$\rightarrow
\left\{\young(124,23,34), \, \young(\onep1\threep,\onep\twop,1\twop)\right\} \rightarrow
\left\{\young(124,23,34,4), \, \young(\onep1\threep,\onep\twop,1\twop,3)\right\} \rightarrow
\left\{\young(123,234,34,4), \, \young(\onep1\threep,\onep\twop3,1\twop,3)\right\}$

\end{example}

\begin{theorem} \label{main}
The Primed-Recording Edelman-Greene map is a double weight preserving bijection: $RSIF_k(\omega) \Rightarrow (P,Q)$,  where $P$ is an Edelman-Greene tableau for $\omega$, and $Q$ is a primed tableau of the same shape.  (The double weight of $(P,Q)$ refers to the double weight of $Q$.)
\end{theorem}

\begin{proof}  The proof relies on a basic fact of Edelman-Greene insertion of an  unsigned reduced word $v$: If $v=v_1 \ldots v_s$ is inserted under Edelman-Greene, then $v_r < v_{r+1}$ if and only if the box added to the insertion tableau in the $r^{th}$ step is in a row weakly above the row where a box is added in the $(r+1)^{st}$ step.  To see the map is well-defined:  Certainly $P$ is an Edelman-Greene tableau.  Certainly $Q$ has weakly increasing rows and columns.  Moreover, for each value of $i$, there is at most one $i$ in each column  because of the forward direction of the basic fact.  And there is at most one $i'$ in each row because of the backwards direction of the basic fact.

The inverse is obtained by applying reverse Edelman-Greene insertion to $P$ in the order prescribed by the standardization of $Q$. (The standardization being the standard Young tableau obtained from $Q$ by extending the partial order induced on the boxes of $Q$ by the order of $X'$ and then using the following rule: If box $b$ and box $b'$ both contain $i$ then $b<b'$ if and only if $b$ lies in a column to the left of $b'$.  If box $b$ and box $b'$ both contain $i'$ then $b<b'$ if and only if $b$ lies in a row above $b'$.) This produces an element of $A_n$.  Now, to make it a signed  factorization, its subdivisions and the signs on the indices are then added in the unique way such that the resulting factorization has the same double weight as $Q$.  Again, the basic fact implies that this inverse is well-defined.

\end{proof}

It now immediately follows from Theorem \ref{main} that:

\begin{theorem}\label{expa}
\begin{eqnarray*}
F_{\omega}^{d}(\mathbf{x},\mathbf{y})= \sum_{T \in E(\omega)} R_{sh(T)}(\mathbf{x},\mathbf{y}),
\end{eqnarray*}
where $E(\omega)$ is the set of Edelman-Greene tableaux for $\omega$ and $sh(T)$ denotes the shape of $T$.
\end{theorem}

\section{Crystal Structure and Schur expansion}

An \defn{(abstract) bicrystal} of type $A_{k-1}$ is a nonempty set $B$ together with the maps
\begin{equation}
\begin{split}
	e_i, f_i &: B \rightarrow B \cup \{0\} \\
	e_{\bar{i}}, f_{\bar{i}} &: B \rightarrow B \cup \{0\} \\
       dw &: B \rightarrow \Lambda \times \Lambda
\end{split}
\end{equation}
where $\Lambda = \mathbb{Z}^{k}_{\geqslant 0}$ is the weight lattice of the root of type $A_{k-1}$ and $1 \leq i \leq k-1$.  Denote by $\alpha_i =\epsilon_i -\epsilon_{i+1}$ for the simple
roots of type $A_{k-1}$, where $\epsilon_i$ is the $i$-th standard basis vector of $\mathbb{Z}^{k}$. Then we require:
\begin{itemize}
\item[\textbf{A1.}]  For $b,b'\in B$, we have $f_ib=b'$ if and only if $b=e_ib'$. In this case $dw(b') =dw(b)-(0,\alpha_i)$.
\item[\textbf{A2.}]  For $b,b'\in B$, we have $f_{\bar{i}}b=b'$ if and only if $b=e_{\bar{i}}b'$. In this case $dw(b') =dw(b)-(\alpha_i,0)$.
\end{itemize}

In this section we induce (via Theorem \ref{main}) an $A_{k-1}$ ($k$ as before) bicrystal structure on the set of reduced signed increasing factorizations of $\omega$ by explicitly defining crystal operators on the set of primed tableaux.  This structure is isomorphic to the bicrystal structure on pairs of SSYT obtained by doubling the usual crystal structure on SSYT (see \cite{Bump.Schilling.2017}). As a result, we will obtain an expansion of $F_{\omega}^d$ as a product of Schur functions corresponding to certain highest weight primed tableaux.

 Given a primed tableau $T$, we will define the \emph{reading word} of $T$ to be the word, composed  only of unprimed entries, obtained by reading the unprimed numbers by row, left to right, moving from bottom to top. Throughout the remainder of the section, we will set $j=i+1$.  The $i-j$ subword of a word is defined to be the word of $i$s and $j$s obtained by erasing all other entries from the word.  The $i-j$ bracketing on the $i-j$ subword is defined as usual in type A.  Each step in the ordering $1'<1<2'<2\cdots$ will be considered a half unit.  In particular, we say that $i$ is obtained from $i'$ by increasing by a half unit, $j'$ is obtained from $i$ by increasing by a half unit, and $j$ is obtained from $j'$ by increasing by a half unit. Finally, if $p$ is a position in $T$, we will write $c(p)$ to mean the content of $p$.  By convention if position $p$ does not describe a filled box of $T$ we take $c(p)= \infty$.

First we define operators $f_i$ on the set of primed tableaux as follows.
\begin{definition}[\textbf{$f_i$ operator}]

First, if the $i-j$ reading subword of $T$ has no unbracketed $i$s, then $f_i(T)=0$.  Otherwise let $x$ denote the position in $T$ corresponding to the rightmost unbracketed $i$ in the $i-j$ reading subword of $T$. $f_i(T)$ is obtained from $T$  by increasing $c(x)$ and $c(q)$ by a half unit for some box $q$, determined as follows: Denote the position immediately to the right of $x$ as $E_x$ and the position immediately below it as $S_x$: 

\begin{itemize}

\item[F1:] { If $c(E_x)\geq j$ and $c(S_x) > j$, set $q=x$.}

\item[F2:] {If $c(E_x)=j'$, set $q=E_x$. }

\item[F3:] {If $c(E_x) \geq j$  and $c(S_x)=j'$ or $c(S_x)=j$, consider the maximal ribbon beginning on $S_x$ and extending in the South and/or West directions which contains only $j$s and $j'$s. Let $q$ be the Southwest-most position of this ribbon.}

\end{itemize}

\end{definition}

\begin{lemma}
If the hypothesis of case F2 holds, then $c(S_q)>j$. If the hypothesis of case F3 holds, then for each $j$ in the ribbon, there is an $i$ in the box diagonally above and to the left of this $j$, and $c(q)=j'$.
\end{lemma}

\begin{proof}
For the first statement, note that we cannot have $c(S_q)=j'$ as this would imply $c(S_x)<j'$, which contradicts $c(S_x)>i=c(x)$.  Moreover, if we have $c(S_q)=j$, then let $k$ denote the number of $j$s in this row.  In order for the $i$ in $x$ to be unbracketed the row above must have at least $k+1$ $i$s.  This implies its leftmost two $i$s  lie above entries less than $j$.  The only such possibility is $j'$ which would contradict the condition that there is at most one $j'$ in each row.

Now to the second statement. In order for the $i$ in position $x$ to be unbracketed, every $j$ in the ribbon must be bracketed with an $i$ which appears in between $q$ and $x$ in the reading word order.  The only way to fit all the necessary $i$s is as described in the statement of the lemma.  Now, if $c(q) \neq j'$ then $c(q)=j$ which implies the box diagonally above $q$ contains an $i$ by the previous sentence.  But this implies the entry to the left of $q$ is $j$ or $j'$ which contradicts the maximality of the ribbon.
\end{proof}

\begin{lemma}
If $f_i(T)\neq0$ then either $q=x$, in which case $f_i$ changes an $i$ in $x$ to a $j$, or else $f_i$ changes an $i$ in $x$ to $j'$ and a $j'$ in $q$ to $j$. Moreover, $f_i(T)$ is a valid primed tableau.
\end{lemma}

\begin{proof}
This is immediate for F1.  For the other two cases it follows from the previous lemma.
\end{proof}

Now we define operators $e_i$ on the set of primed tableaux as follows.
\begin{definition}[\textbf{$e_i$ operator}]

First, if the $i-j$ reading subword of $T$ has no unbracketed $j$s, then $e_i(T)=0$.  Otherwise let $y$ denote the position in $T$ corresponding to the leftmost unbracketed $j$ in the $i-j$ reading subword of $T$. $e_i(T)$ is obtained from $T$  by decreasing $c(y)$ and $c(p)$ by a half unit for some box $p$, determined as follows: Denote the position immediately to the left of $y$ as $W_y$ and the position immediately above it as $N_y$: 

\begin{enumerate}

\item{ If $c(W_y)\leq i$ and $c(N_y) < i$, set $p=y$.}

\item{ If $c(W_y)=j'$, set $p=W_y$. }

\item{ If $c(W_y) \leq i$  and $c(N_y)=i$ or $c(N_y)=j'$, consider the maximal ribbon beginning on $N_y$ and extending in the North and/or East directions which contains only $i$s and $j'$s. Let $p$ be Northeast-most position of this ribbon.}

\end{enumerate}

\end{definition}

Symmetric arguments to those given above show:

\begin{lemma}
If $e_i(T)\neq0$ then either $p=y$, in which case $e_i$ changes a $j$ in $y$ to an $i$, or else $e_i$ changes an $j$ in $y$ to $j'$ and a $j'$ in $p$ to $i$. Moreover, $e_i(T)$ is a valid primed tableau.
\end{lemma}

In fact $f_i$ and $e_i$ so defined are inverses:

\begin{proposition}\label{inv}
If $f_i(T) \neq 0$ then $e_i(f_i(T))=T$.  Similarly, if $e_i(T) \neq 0$ then  $f_i(e_i(T))=T$.
\end{proposition}

\begin{proof}
Suppose that $f_i(T)$ is obtained from $T$ by case $C$ ($C \in \{1,2,3\}$) of the definition of $f_i$, with $x$ and $q$ representing the positions selected as in the definition.  Then it is not difficult to see that $e_i(f_i(T))$ can be obtained from $f_i(T)$ by case $k$ in the definition of $e_i$ with positions $y=q$ and $p=x$. Since a half unit has been added and then subtracted from positions $x$ and $q$ it is clear that  $e_i(f_i(T))=T$. The second statement is proved symmetrically. 
\end{proof}

We now define operators $f_{\bar{i}}$ and $e_{\bar{i}}$ on the set of primed tableaux and on words.  On words these operators act as the usual type A operators restricted to primed entries. To describe their action on primed tableaux, we use a special type of transposition.

\begin{definition}
Given a primed tableau, $T$, we define $T^{+}$ to be obtained by transposing $T$ and then adding a half unit to each entry.  Similarly we define $T^{-}$ to be obtained by transposing $T$ and then subtracting a half unit from each entry.  Applied to a word $w$, $w^+$ and $w^-$ indicate  simply adding a half unit and subtracting a half unit from each entry respectively.
\end{definition}
Clearly, we have $(T^{+})^{-}=T$ and $(w^+)^-=w$. 
\begin{definition}
We define $f_{\bar{i}}(T)=(f_i(T^+))^-$ and $e_{\bar{i}}(T)=(e_i(T^+))^-$ for  any primed tableau, $T$.  Similarly we define $f_{\bar{i}}(w)=(f_i(w^+))^-$ and $e_{\bar{i}}(w)=(e_i(w^+))^-$ for any primed word, $w$.
\end{definition}

\begin{figure}

\begin{center}

\begin{tikzpicture}[>=latex,line join=bevel,]

 \node (node_1) at (0bp,0bp) [draw,draw=none] {\young(\onep1,1\twop)};
 \node (node_2) at (70bp,0bp) [draw,draw=none] {\young(\onep\twop,12)};
 \node (node_3) at (140bp,30bp) [draw,draw=none] {\young(\onep\twop,22)};
 \node (node_4) at (140bp,-30bp) [draw,draw=none] {\young(\onep\twop,13)};
 \node (node_5) at (210bp,0bp) [draw,draw=none] {\young(\onep\twop,23)};
 \node (node_6) at (280bp,0bp) [draw,draw=none] {\young(\onep\twop,33)};

\draw [solid,blue,->] (node_1) -- (node_2) node[midway,above,black] {$f_1$};
\draw [solid,blue,->] (node_2) -- (node_3) node[midway,above,black] {$f_1$};
\draw [solid,red,->] (node_2) -- (node_4)node[midway,above,black] {$f_2$};
\draw [solid,red,->] (node_3) -- (node_5)node[midway,above,black] {$f_2$};
\draw [solid,blue,->] (node_4) -- (node_5);
\draw [solid,red,->] (node_5) -- (node_6) node[midway,above,black] {$f_2$};
  \node (node_11) at (0bp,-100bp) [draw,draw=none] {\young(\onep1,1\threep)};
 \node (node_12) at (70bp,-100bp) [draw,draw=none] {\young(\onep2,1\threep)};
 \node (node_13) at (140bp,-70bp) [draw,draw=none] {\young(\onep2,2\threep)};
 \node (node_14) at (140bp,-130bp) [draw,draw=none] {\young(\onep\threep,13)};
\node (node_15) at (210bp,-100bp) [draw,draw=none] {\young(\onep\threep,23)};
\node (node_16) at (280bp,-100bp) [draw,draw=none] {\young(\onep\threep,33)};

\draw [solid,blue,->] (node_11) -- (node_12) node[midway,above,black] {$f_1$};
\draw [solid,blue,->] (node_12) -- (node_13) ;
\draw [solid,red,->] (node_12) -- (node_14) ;
\draw [solid,red,->] (node_13) -- (node_15) ;
\draw [solid,blue,->] (node_14) -- (node_15);
\draw [solid,red,->] (node_15) -- (node_16) node[midway,above,black] {$f_2$};


 \node (node_21) at (0bp,-200bp) [draw,draw=none] {\young(11,\twop\threep)};
 \node (node_22) at (70bp,-200bp) [draw,draw=none] {\young(12,\twop\threep)};
 \node (node_23) at (140bp,-170bp) [draw,draw=none] {\young(\twop2,2\threep)};
 \node (node_24) at (140bp,-230bp) [draw,draw=none] {\young(1\threep,\twop3)};
 \node (node_25) at (210bp,-200bp) [draw,draw=none] {\young(\twop\threep,23)};
 \node (node_26) at (280bp,-200bp) [draw,draw=none] {\young(\twop\threep,33)};

\draw [solid,blue,->] (node_21) -- (node_22) node[midway,above,black] {$f_1$};
\draw [solid,blue,->] (node_22) -- (node_23) ;
\draw [solid,red,->] (node_22) -- (node_24) node[midway,above,black] {$f_2$};
\draw [solid,red,->] (node_23) -- (node_25) node[midway,above,black] {$f_2$};
\draw [solid,blue,->] (node_24) -- (node_25) node[midway,above,black] {$f_1$};
\draw [solid,red,->] (node_25) -- (node_26)  node[midway,above,black] {$f_2$};
 \draw [dashed,blue,->] (node_1) -- (node_11) node[midway,left,black] {$f_{\bar{1}}$};
 \draw [dashed,blue,->] (node_2) -- (node_12) node[midway,left,black] {$f_{\bar{1}}$};
 \draw [dashed,blue,->] (node_3)  to [out=300,in=30] (node_13) ;
 \draw [dashed,blue,->] (node_4) to [out=240,in=150] (node_14) ;
 \draw [dashed,blue,->] (node_5) -- (node_15) node[midway,left,black] {$f_{\bar{1}}$};
 \draw [dashed,blue,->] (node_6) -- (node_16) node[midway,left,black] {$f_{\bar{1}}$};
 
 \draw [dashed,red,->] (node_11) -- (node_21) node[midway,left,black] {$f_{\bar{2}}$};
 \draw [dashed,red,->] (node_12) -- (node_22) node[midway,left,black] {$f_{\bar{2}}$};
 \draw [dashed,red,->] (node_13)  to [out=300,in=30] (node_23);
 \draw [dashed,red,->] (node_14) to [out=240,in=150] (node_24);
 \draw [dashed,red,->] (node_15) -- (node_25) node[midway,left,black] {$f_{\bar{2}}$};
 \draw [dashed,red,->] (node_16) -- (node_26) node[midway,left,black] {$f_{\bar{2}}$};

\end{tikzpicture}
\caption{A connected component of the $A_2$ bicrystal for $\lambda=(2,2)$. A few labels are omitted to prevent cluttering.}

\label{fig:}

\end{center}

\end{figure}
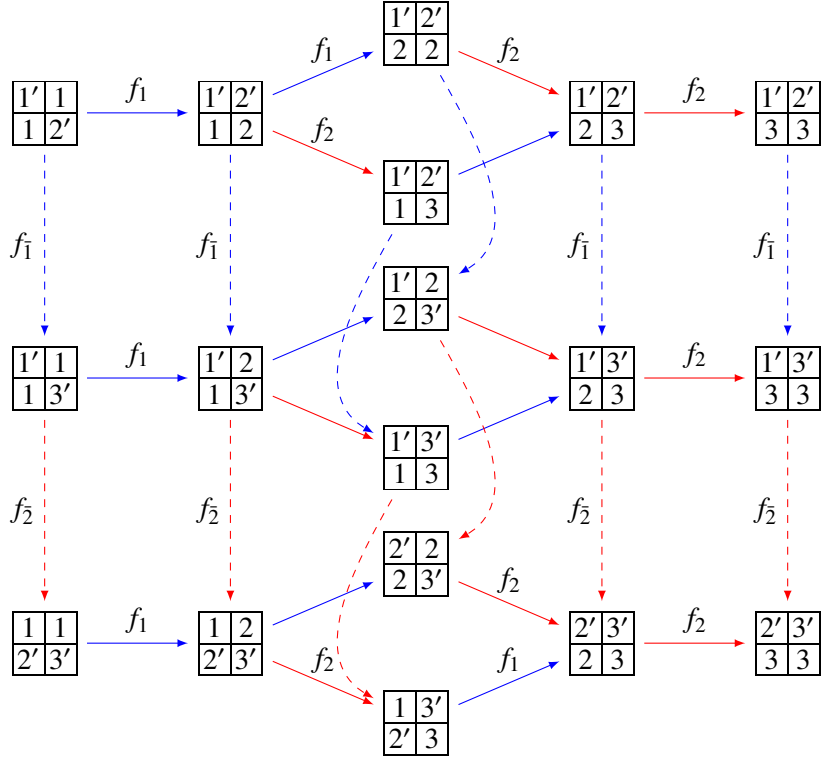

\subsection{Proof that the operators $f_i,e_i,f_{\bar{i}},e_{\bar{i}}$ form an $A_{k-1}$ bicrystal}
We consider words in the alphabet $X_k'$.  For such a word $w$, The \emph{double weight} of $w$, denoted $(dw(w,1),dw(w,2))$ is the pair $(X,Y)$, where the $i^{th}$ entry of $X$ records the number of primed $i$s in $w$, and  the $i^{th}$ entry of $Y$ records the number of unprimed $i$s in $w$.  We have a natural bicrystal structure on such words: $f_i(w)$ and $e_i(w)$ refer to the usual crystal operators on words (e.g., \cite{Bump.Schilling.2017}) restricted to the nonprimed letters of $w$, while $f_{\bar{i}}(w)$ and $e_{\bar{i}}(w)$ refer to the usual crystal operator on words (e.g., \cite{Bump.Schilling.2017}) restricted to the primed letters of $w$. (The operators never change whether a letter is primed or not.)

Next we consider the (nonshifted) mixed Haiman insertion of a word $w$, denote $\mathcal{P}(w)$, for a word $w$ from the alphabet $X_k'$. Given a word $w=w_1 w_2 \ldots w_h$ we recursively construct a sequence of tableaux
$\emptyset=T_0,T_1,\ldots,T_h$.
To obtain the tableau $T_s$, insert the letter $w_s$ into $T_{s-1}$ as follows.
First, if $w_s$ is unprimed, insert $w_s$ into the top row of $T_{s-1}$, bumping out the leftmost element $y$ that is strictly greater
than $w_s$. If $w_s$ is primed, insert $w_s$ into the leftmost column of $T$, bumping out the uppermost element $y$ that is strictly greater
than $w_s$.
\begin{enumerate}
	\item If $y$ is not primed, then insert it into the next row below, bumping out the leftmost
   	element that is strictly greater than $y$ from that row.
	\item If $y$ is primed, then insert it into the next column to the right, bumping out
	the uppermost element that is strictly greater than $y$ from that column.
\end{enumerate}
Continue until an element is inserted which is greater than or equal to all elements in the row/column it is inserted into, and append this element to the end of this row. We define the recording tableau for such an insertion to be the standard Young tableau that records the sequence of shapes of the $T_i$.

Note that it follows from the symmetry in the definition of Haiman insertion that $\mathcal{P}(w^+)=(\mathcal{P}(w))^+$ and $\mathcal{P}(w^-)=(\mathcal{P}(w))^-$.

\begin{theorem} \cite{Haiman.1989} \label{hm}
Haiman insertion gives a double weight preserving bijection from words from $ X_k'$  to pairs of tableaux $(T,S)$, where $T$ is a primed tableau and $S$ is a standard Young tableau of the same shape.
\end{theorem}

 Our next goal is to show that $f_i(\mathcal{P}(w))=\mathcal{P}(f_i(w))$ (where the lefthand side involves the operator defined in this paper and the right hand side involves the usual crystal operator on primed words restricted to the nonprimed entries).  To do this we consider Haiman insertion step by step.

Suppose $w_1 \ldots w_h$ is a primed word and that $T_i$ are the tableaux as in the definition of Haiman insertion. For each $s$ define the reading word of $(T_{s-1},w_s \cdots w_h)$ to be the reading word of $T$ concatenated with the unprimed entries of $w_s \cdots w_h$.  $f_i$ acts on $(T_{s-1},w_s \cdots w_h)$ by selecting the rightmost unbracketed $i$ in its reading word.  If this is an element of $w_s \cdots w_h$, it simply changes this $i$ to $j$.  Otherwise, $f_i$ acts according to the rules for primed tableaux specified earlier.

\begin{lemma}\label{comstep}
Let $(T_{s},w_{s+1}\cdots w_h)$ denote the result of applying one insertion step to $(T_{s-1},w_{s}\cdots w_h)$.
Then set $f_i(T_{s-1},w_{s}\cdots w_h)=(T_{s-1}^*,w_{s}^*\cdots w_h^*)$.
Finally, let $(T_{s}^*,w_{s+1}^*\cdots w_h^*)$ denote the result of applying one insertion step to $(T_{s-1}^*,w_{s}^*\cdots w_h^*)$.
Then we have $(T_{s}^*,w_{s+1}^*\cdots w_h^*)=f_i(T_{s},w_{s+1}\cdots w_h)$.
\end{lemma}
\begin{proof}

First of all, if  $f_i$ acts on some $w_k$ for $k>s$ (or is $0$) then since (as is easily verified) the reading word of $(T_{s-1},w_s)$ and the reading word of $T_s$ have the same number of bracketed $i$s, it follows that $f_i$ acts on $(T_s,w_{s+1}\cdots w_h)$ at $w_k$ (or is $0$) as well and the result follows.  Hence we may assume $h=s$ and $f_i(T_{s-1},w_{s}\cdots w_h)\neq 0$.

Next, since the $f_i$ operator only concerns $i$s, $j'$s, and $j$s, it suffices ``ignore" entries smaller than $i$ or greater than $j$.  In other words it suffices to prove the following:  Let $T$ be a primed tableau of arbitrary skew shape composed only of $i$s, $j'$s, and $j$s.  Then if $z$ is some inner corner box of $T$, and $V_z(T)$ denotes removing the entry in box $z$ and continuing mixed Haiman insertion from this point as if the entry of $z$ had been bumped, we have the following equality: $f_i(V_z(T))=V_z(f_i(T))$, where the operator $f_i$ is applied to a skew tableau just as for a straight shape tableau.

\begin{example}
Set $T=\young(:\twor2,2\threep3,\threep)$ and let $z$ be the box containing the red entry.  To compute $V_z(T)$, we remove the red $2$, insert it in the $2^{nd}$ row, bumping a $3'$, which is inserted in the $3^{rd}$ column, bumping a $3$, which is appended to the $3^{rd}$ row. Thus $V_z(T)=\young(::2,22\threep,\threep3)$. Then $f_2$ acts on $V_z(T)=\young(::\twobl,22\threepg,\threep3)$ via F3 at the blue and green entries in boxes $x$ and $q$ to give $f_2(V_z(T))=\young(::\threep,223,\threep3)$.

On the other hand $f_2$ acts on $T=\young(:2\twobl,2\threepg3,\threep)$ via F3 at the blue and green entries in boxes $x$ and $q$ to give $f_2(T)=\young(:2\threep,233,\threep)$.  Then, $V_z(f_2(T))$ is computed by removing the red entry from $f_2(T)=\young(:\twor\threep,233,\threep)$, inserting it into the $2^{nd}$ row, bumping a $3$,  which is appended to the $3^{rd}$ row to give $V_z(f_2(T))=\young(::\threep,223,\threep3)$.
\end{example}

As before, let $x$ and $q$ denote the position of $T$ in which $f_i$ acts.  Let $r_x$, $r_q$, and $r_z$ denote the rows containing $x$, $q$, and $z$ respectively. For each possible way in which $f_i(T)$ could be formed from $T$, we explain how $f_i$ acts on $V_z(T):=T'$.  It is left to the reader to see that applying $V_z$ to $f_i(T)$ has the same result.

\begin{enumerate}
\item{Suppose $f_i$ acts on $T$ by F1 at $x$.
\begin{enumerate}
\item{If $r_z>r_x$ then $f_i$ acts on $T'$ by F1 at $x$ unless $z$ is in the column to the left of $x$ and  $c(z)=j'$.  In this case $f_i$ acts on $T'$ by F3 at $x$ and $S_x$.}
\item{If $r_z=r_x$ the row $r_x+1$ contains no $j$ (otherwise, the $i$ in $r_x$ would be bracketed).  If it also contains no $j'$, then $f_i$ acts by F1 on $T'$ at the last entry of row $r_x+1$.  Otherwise $f_i$ acts by F2 at the last two entries of row $r_x+1$.}
\item{If $r_z=r_x-1$, then $f_i$ acts on $T'$ by F1 at $x$ unless $c(z)=i$.  In this case we must have $c(E_x)=j$ (or else the $i$ in $z$ would be unbracketed), so that  $f_i$ acts on $T'$ by F1 at $E_x$.}
\item{If $r_z<r_x-1$ then $f_i$ acts on $T'$ by F1 at $x$.}

\end{enumerate}}
\item{Suppose $f_i$ acts on $T$ by F2 at $x$ and $E_x$.
\begin{enumerate}
\item{If $r_z>r_x$ then $f_i$ acts on $T'$ by F2 at $x$ and $E_x$.}

\item{If $r_z=r_x$ the row $r_x+1$ contains no $j$ (otherwise, the $i$ in $r_x$ would be bracketed).  If it also contains no $j'$, then $f_i$ acts by F1 on $T'$ at the last entry of row $r_x+1$.  Otherwise $f_i$ acts by F2 at the last two entries of row $r_x+1$.}
\item{If $r_z=r_x-1$, then $f_i$ acts on $T'$ by F2 at $x$ and $E_x$ unless $c(z)=i$.  In this case we must have a $j$ to the right of $E_x$ (or else the $i$ in $z$ would be unbracketed), so that  $f_i$ acts on $T'$ by F2 at $E_x$ and the position to the right of $E_x$.}

\item{If $r_z<r_x-1$ then $f_i$ acts on $T'$ by F2 at $x$ and $E_x$.}
\end{enumerate}}

\item{Suppose $f_i$ acts on $T$ by F3 at $x$ and $q$.
\begin{enumerate}

\item{If $r_z>r_q$ then $f_i$ acts on $T'$ by F3 at $x$ and $q$ unless $z$ is in the column to the left of $x$ and  $c(z)=j'$.  In this case $f_i$ acts on $T'$ by F3 at $x$ and $S_q$.}

\item{If $r_z=r_q$, then $q=E_z$ and $c(z)=i$. If the row $r_q+1$ contains a $j'$ in the column to the left of $q$ then  $f_i$ acts on $T'$ by F3 at $x$ and $S_q$.  Otherwise $f_i$ acts on $T'$ by F3 at $x$ and $q$.}

\item{If $r_z=r_q-1$, then $z$ is above $q$ with $c(z)=i$.  If $z \neq x$ then $f_i$ acts on $T'$ by F3 at $x$ and $E_q$.  Otherwise $f_i$ acts on $T'$ by F1 at $q$.}

\item{If $r_z<r_q-1$, and $r_z>r_x$, then $f_i$ acts by F3 at $x$ and $q$ (but one of the boxes in the ribbon defined in F3 has moved diagonally down and to the right).}
\item{If $r_z<r_q-1$, and $r_z=r_x+1$. Then $c(E_x)=j$ and $f_i$ acts on $T'$ by F3 at $E_x$ and $q$. }
\item{If $r_z<r_q-1$, and $r_z=r_x$, then $f_i$ acts by F3 at $S_x$ and $q$.}
\item{If $r_z<r_q-1$ and $r_z<r_x+1$ then $f_i$ acts on $T'$ by F3 at $x$ and $q$.}

\end{enumerate}}

\end{enumerate}
\end{proof}

It is immediate from this definition and Proposition \ref{inv} that $e_{\bar{i}}$ and $f_{\bar{i}}$ are also inverses.  Moreover, the relationship to Haiman insertion is the same as for the non-barred operators:

\begin{theorem}
$\mathcal{P}(f_{i}(w))=f_{i}(\mathcal{P}(w))$ and $\mathcal{P}(f_{\bar{i}}(w))=f_{\bar{i}}(\mathcal{P}(w))$.
\end{theorem}

\begin{proof}
The first equality is immediate from Theorem \ref{comstep}.  For the second, note that:
\begin{eqnarray*}
\mathcal{P}(f_{\bar{i}}(w))=\mathcal{P}((f_i(w^+))^-)=(\mathcal{P}(f_i(w^+)))^-=(f_i(\mathcal{P}(w))^+)^-=f_{\bar{i}}(\mathcal{P}(w)).
\end{eqnarray*}

\end{proof}
From this and Theorem \ref{hm} it follows that Haiman insertion is a bicrystal isomorphism, hence proving that the crystal operators do in fact give a type $A$ bicrystal on primed tableaux.  From this we conclude:
\begin{theorem}
For any $\omega \in A_n$:
\begin{eqnarray*}
F_{\omega}^{d}(\mathbf{x},\mathbf{y})= \sum_{T \in E(\omega)} \sum_{S \in \mathcal{H}(sh(T))} s_{dw(S,1)}(\mathbf{x})s_{dw(S,2)}(\mathbf{y}),
\end{eqnarray*}
where $\mathcal{H}(sh(T))$ denotes all primed tableaux $S$ of the same shape as $T$ such that the reading word of $S$ and $S^+$ are both reverse Yamanouchi words.
\end{theorem}

Next we give an example of the theorem:

\begin{example}
If $\omega=121$, then 
\begin{eqnarray*}
E(\omega)=\left\{\young(12,2)\right\}
\end{eqnarray*}
\begin{eqnarray*}
\mathcal{H}(2,1)=\left\{\young(11,2),\young(\onep1,1),\young(\onep1,2),\young(\onep1,\onep),\young(\onep\twop,1),\young(\onep\twop,\onep)\right\}
\end{eqnarray*}
\begin{eqnarray*}
F_{\omega}^{d}(\mathbf{x},\mathbf{y})=  s_{21}(\mathbf{x})+s_{2}(\mathbf{x})s_1(\mathbf{y})+s_{11}(\mathbf{x})s_1(\mathbf{y})+
s_{1}(\mathbf{x})s_2(\mathbf{y})+s_{1}(\mathbf{x})s_{11}(\mathbf{y})+s_{21}(\mathbf{y})
\end{eqnarray*}
\end{example}

\section{Primed Signed Tableaux}

In this section we prove that primed tableaux and signed tableaux are in bijective correspondence.  This will then allow us to investigate the algebraic relationship between the type $A$ Stanley symmetric function and the double Stanley symmetric function (see the end of the section).

 If $T \in PST(\lambda/\mu,X,Y,j)$, then we may obtain an element of $PST(\lambda/\mu,X,Y,j-1)$ if $j>0$ by applying the following \emph{inward conversion} \footnote{A somewhat similar definition appears in  \cite{Haiman.1989} for the case where letters may not be repeated.  In fact, one can use a certain standardization process along with Haiman's mixed insertion, Haiman's conversion, and Haiman's Theorem 3.12, \cite{Haiman.1989} to derive Theorem \ref{conv} below for the specific case of $\mu=\emptyset$, i.e., for straightshape tableaux.  However, the proofs needed to do this are somewhat more complicated than those employed below, and the result, of course, less general.}  procedure to $T$ $X(j)$ times:

\begin{enumerate}
\item{Change the uppermost primed $j$ in $T$ to a barred $j$.}
\item{Repeat the following procedure until all rows and columns are weakly increasing: Switch the lowermost barred $j$ with either the entry above it or to its left, determined as follows:
\begin{itemize}
\item{If only one of the entries exists, take it.}
\item{If these entries are not equal, take the larger.}
\item{If they are equal and are unmarked, take the one above.}
\item{If they are equal and are marked, take the one on the left.}
\end{itemize}}
\end{enumerate}

It is not immediately clear that the result is in fact a $PST$--namely it seems possible that the result of the algorithm above may contain two or more barred $j$s in the same row. To see this is impossible we reason as follows: For each of the $X(j)$ conversions, let the corresponding \emph{conversion path} be the set of boxes which are altered during this conversion.  If $p_1, \ldots, p_{X(j)}$ denote these paths, then it is not difficult to check the following.

\begin{lemma}
Let $1 \leq i <X_{j}$.  If $b$ is the highest box in its column that belongs to $p_i$ then neither $b$ nor any box above $b$ in this column belongs to $p_{i+1}$.  Further, any box in $p_{i+1}$ that lies to the left of the upper leftmost box of $p_i$ also lies below it.
\end{lemma}

This implies there will never be more than one barred $j$ in any row during the inward conversions.

\begin{example}
Applying inward conversion twice to $\young(\fourb\threeb14,\fourb\onep\twop,\threeb1\twop,22) $   yields $\young(\fourb\threeb\twob4,\fourb\twob\onep,\threeb11,22) $.

\end{example}

Similarly, if $T \in PST(\lambda/\mu,X,Y,j)$, then we may obtain an element of $PST(\lambda/\mu,X,Y,j+1)$ if $j<k$ by the \emph{outward conversion} procedure to $T$ $X(j+1)$ times:

\begin{enumerate}
\item{Change the lowermost barred $j$ in $T$ to a primed $j$ if it exists.}
\item{Repeat the following procedure until all rows and columns are weakly increasing: Switch the uppermost primed $j$ with either the entry below it or to its right, determined as follows:
\begin{itemize}
\item{If only one of the entries exists, take it.}
\item{If these entries are not equal, take the smaller.}
\item{If they are equal and are unmarked, take the one below.}
\item{If they are equal and are marked, take the one on the right.}
\end{itemize}}
\end{enumerate}

Analogously, we have: If $q_1, \ldots, q_{X(j)}$ denote the conversion paths, then it is not difficult to check the following.

\begin{lemma}
Let $1 \leq i <X_{j}$.  If $b$ is the lowest box in its column which belongs to $q_i$ then neither $b$ nor any box below $q$ in this column belongs to $q_{i+1}$.  Further, any box in $q_{i+1}$ that lies to the right of the lower rightmost box of $q_i$ also lies above it.
\end{lemma}

This implies there will never be more than one barred $j$ in any column during the outward conversions.

This construction leads to the major result of this section:

\begin{theorem} \label{conv}
Fix $\lambda$, $\mu$, $X$, and $Y$.  Then for  any $1 \leq j \leq k$ there is a bijection  $PST(\lambda/\mu,X,Y,j) \Rightarrow PST(\lambda/\mu,X,Y,j-1)$.
\end{theorem}

\begin{proof}
The bijection is given by applying inward conversion $X(j)$ times (and the inverse by applying outward conversion $X(j)$ times).
\end{proof}

\begin{definition}
Let $X, Y \in \mathbb{Z}_{\geq 0}^k$.

A  \emph{signed tableau} of shape $\lambda/\mu$ and double weight $(X,Y)$ is an element of  $PST(\lambda,\mu,X,Y,0)$.

\end{definition}

\begin{corollary} \label{tabi}
Letting $PT(\lambda/\mu)$ denote the set of all primed tableaux of shape  $\lambda/\mu$ and $ST(\lambda/\mu)$ denote the set of all signed tableaux of shape  $\lambda/\mu$, there is a double weight preserving bijection: $PT(\lambda/\mu) \Rightarrow ST(\lambda/\mu)$.
\end{corollary}

Given any symmetric function $F$ we may consider its value on the doubled set of variables $(\mathbf{x},\mathbf{y})$, denoted simply $F(\mathbf{x},\mathbf{y})$.  Applying the involution on the ring of symmetric functions $\omega$ defined by condition $\omega(s_{\lambda}(\mathbf{x}))=s_{\lambda'}(\mathbf{x})$ to the function $F(\mathbf{x},\mathbf{y})$ considered as function in the variable $\mathbf{x}$ (over the ring of symmetric functions in $\mathbf{y}$) yields a symmetric function which we denote $F(\mathbf{x}/\mathbf{y})$.  Since it follows from \ref{tabi} that
\begin{eqnarray*}
R_{\lambda/\mu}(x,y):=\sum_{T \in PT(\lambda/\mu)} x^{dw(T,1)}y^{dw(T,2)}=\sum_{T \in ST(\lambda/\mu)} x^{dw(T,1)}y^{dw(T,2)},
\end{eqnarray*}
it is clear from the definition of $ST(\lambda/\mu)$ that we have $R_{\lambda/\mu}(\mathbf{x},\mathbf{y})=s_{\lambda/\mu}(\mathbf{x}/\mathbf{y})$.  From this and Theorem \ref{expa} it follows that for all $\omega \in A_n$, $F_{\omega}^d(\mathbf{x},\mathbf{y})=F_{\omega}^A(\mathbf{x}/\mathbf{y})$.  In particular this implies a relationship first noted by Lam (\cite{Lam.1995}): $F_{\omega}^C(\mathbf{x})=F_{\omega}^d(\mathbf{x},\mathbf{x})=F_{\omega}^A(\mathbf{x}/\mathbf{x})$ for $\omega \in A_n$.


\section{Conjectures for type $C$}

Recall that the concept of $RSIF_k$ was defined for any $\omega \in C_{n+1}$, not just $\omega \in A_n$.  However, we only defined $F_{\omega}^d(\textbf{x},\textbf{y})$ for elements of $A_n$.  This is because, unfortunately, generalizing  $F_{\omega}^d(\textbf{x},\textbf{y})$ in the canonical manner to $C_{n+1}$ does not result in a symmetric function.  There are some cases however, where $F_{\omega}^d(\textbf{x},\textbf{y})$ does exhibit some nice symmetry properties when we generalize to elements outside of $A_n$.  The approaches needed to investigate this symmetry are unrelated to those in the rest of the paper and are relatively technical so we will not include them in detail, but rather give a short overview of some constructions and conjectures.\\

\begin{definition}

We say an element $\omega\in C_{n+1}$ is \defn{unknotted}\footnote{The conjectures of this section are theorems if we replace unknotted with the slightly stronger concept of \emph{untangled}.  An element $\omega \in C_{n+1}$ is \defn{untangled} if for all reduced words $w$ for $\omega$:

\begin{enumerate}

\item{$s_2$ does not appear in $w$.}

\item{For $i>2$, if $s_{i}$ and $s_{i+1}$ appear in $w$, and one of  $s_{i}$ or $s_{i+1}$ appears more than once, then $s_{i-1}$ and $s_{i+2}$ do not appear in $w$.}

\end{enumerate}

However, we do not provide the proofs here.}  if the following hold for all reduced words $w$ for $\omega$:

\begin{enumerate}

\item{If the sequence $s_0s_1s_0s_1$ appears in $w$, then $s_2$ does not.}

\item{For $i>0$ if the sequence $s_is_{i+1}s_i$ appears in $w$ then $s_{i+2}$ does not.}

\end{enumerate}

\end{definition}

For instance, the permutations corresponding to the reduced words, $s_1s_2s_3s_4s_5s_4$, $s_2s_1s_0s_1s_2s_3s_2$, \\ $s_1s_0s_1s_0s_3s_4s_3$, and $s_2s_1s_3s_2s_4s_3$ are \emph{unknotted}.

\begin{definition}

A \defn{signed Edelman-Greene tableau} for an unknotted signed permutation, $\omega$, is a tableau  composed of entries from the alphabet $\{\cdots<\threeb<\twob<\oneb<0<1<2<3<\cdots\}$ such that: 

\begin{enumerate}

\item{The reading word of $T$, obtained from reading the rows of $T$ left to right, bottom to top and then changing any $i$ or -$i$ to $s_i$ is a reduced word for $\omega$.}

\item{The rows and columns of $T$ are weakly increasing.}

\item{Whenever  $T_{ij}=T_{(i+1)j}$ then we have one of:

\begin{itemize}

\item There exists $k>j$ such that $|T_{(i+1)k}|=|T_{ij}|+1$.

\item  There exists $l<j$ such that $|T_{il}|=|T_{ij}|+1$.

\item $T_{ij}=\oneb=T_{(i+1)j}$ and $T_{i(j+1)}=0=T_{(i+1)(j+1)}$.
\end{itemize} }
\end{enumerate}

\end{definition}

\begin{example}

The following is a signed Edelman-Greene tableau for the unknotted signed permutation $\{{-2},{-1},5,4,3,8,7,6\} \in C_8$: 
\begin{eqnarray*}
\young(\sevb\oneb06,\threeb\oneb06,\threeb4)
\end{eqnarray*}
since its reading word $s_3s_4s_3s_1s_0s_6s_7s_1s_0s_6$ produces the given signed permutation, and each of the four instances of repeated entries in a column is allowed by rule (3).

 \end{example}

We are now ready to state the three conjectures of this section.

\begin{conjecture}\label{con1}

Suppose $\omega \in C_{n+1}$ is unknotted,  Then we have:

\begin{eqnarray*}
F_{\omega}^d(\mathbf{x},\mathbf{x})= \sum_{\lambda} \bar{E}_{\omega}^{\lambda} s_{\lambda}(\mathbf{x})
\end{eqnarray*}
where $\bar{E}_{\omega}^{\lambda}$ is the number of signed Edelman-Greene tableaux for $\omega$ with shape $\lambda$.
\end{conjecture}

\begin{conjecture}\label{con2}
Suppose $\omega \in C_{n+1}$ is unknotted and any reduced word for $\omega$ has at most one $s_0$. Then:
\begin{eqnarray*}
F_{\omega}^{d}(\mathbf{x},-\mathbf{x})=\sum_{r \text{ even}}\sum_{\lambda} \bar{E}_{\omega}^{\lambda r} s_{\lambda}(\mathbf{x})-\sum_{r \text{ odd}} \sum_{\lambda} \bar{E}_{\omega}^{\lambda r} s_{\lambda}(\mathbf{x})
\end{eqnarray*}
where $\bar{E}_{\omega}^{\lambda r}$ is the number of signed Edelman-Greene tableaux for $\omega$ with shape $\lambda$ and exactly $r$ barred entries.
\end{conjecture}

\begin{conjecture}\label{con3}
Suppose $\omega \in C_{n+1}$ is unknotted and any reduced word for $\omega$ has no $s_0$ (i.e., $\omega \in A_n$). Then:
\begin{eqnarray*}
F_{\omega}^d(\mathbf{x},t\mathbf{x})=\sum_{\lambda} \bar{E}_{\omega}^{\lambda r} s_{\lambda}(\mathbf{x})t^r
\end{eqnarray*}
where $\bar{E}_{\omega}^{\lambda r}$ is the number of signed Edelman-Greene tableaux for $\omega$ with shape $\lambda$ and exactly $r$ barred entries.
\end{conjecture}

\begin{example}\label{ex1}
Let $\omega$ be the unknotted signed permutation $\{-2,-1,4,3\}$.  Then we have the following signed Edelman-Greene tableaux for $\omega$: \\
\begin{eqnarray*}
&\young(\threeb\oneb01,0) \,\,\,\,\, \young(\oneb013,0)& \\  
\\
&\young(\threeb\oneb0,\oneb0) \,\,\,\,\, \young(\oneb03,\oneb0) \,\,\,\,\, \young(\threeb01,01) \,\,\,\,\, \young(\oneb01,03) &\\
\\
& \young(\threeb01,\oneb,0) \,\,\,\,\, \young(\oneb01,0,3) \,\,\,\,\,  \young(\threeb\oneb0,0,1) \,\,\,\,\, \young(\oneb03,0,1) &  \\
\\
&\young(\threeb0,01,1) \,\,\,\,\, \young(\oneb0,\oneb0,3) \,\,\,\,\, \young(\oneb0,03,1) \,\,\,\,\, \young(\threeb\oneb,\oneb0,0) &\\
\\
&\young(\threeb0,\oneb,0,1) \,\,\,\,\, \young(\oneb0,0,1,3)& \\ 
\end{eqnarray*}
By conjecture \ref{con1}, $F_{\omega}^d(\mathbf{x},\mathbf{x})=2s_{41}(\mathbf{x})+4s_{32}(\mathbf{x})+4s_{311}(\mathbf{x})+4s_{221}(\mathbf{x})+2s_{2111}(\mathbf{x})$.
 \end{example}

\begin{example} \label{ex2}

Let $\omega$ be the unknotted signed permutation $\{3,2,-1,4\}$.  Then we have the following signed Edelman-Greene tableaux for $\omega$: \\
\begin{eqnarray*}
&\young(\twob012)& \\  
\\
&\young(012,2) \,\,\,\,\, \young(\twob12,0) \,\,\,\,\, \young(\twob\oneb2,0) &\\
\\
& \young(\oneb2,02) \,\,\,\,\, \young(\twob1,\oneb2) &  \\
\\
&\young(\twob1,\oneb,0) \,\,\,\,\, \young(\twob\oneb,\oneb,0) \,\,\,\,\, \young(\oneb2,0,2) &\\
\\
&\young(\twob,\oneb,0,2)& \\ 
\end{eqnarray*}
By conjecture \ref{con1}, $F_{\omega}^d(\mathbf{x},\mathbf{x})=s_{4}(\mathbf{x})+3s_{31}(\mathbf{x})+2s_{22}(\mathbf{x})+3s_{211}(\mathbf{x})+s_{1111}(\mathbf{x})$.  \\
By conjecture \ref{con2}, $F_{\omega}^{d}(\mathbf{x},-\mathbf{x})=-s_4(\mathbf{x})+s_{31}(\mathbf{x})-s_{211}(\mathbf{x})+s_{1111}(\mathbf{x})$.
 \end{example}

\begin{example} \label{ex3}

Let $\omega$ be the unknotted signed permutation $\{3,2,1,4\}$.  Then we have the following signed Edelman-Greene tableaux for $\omega$: \\
\begin{eqnarray*}
&\young(\twob12) \,\,\,\,\, \young(\twob\oneb2) &\\
\\
& \young(12,2) \,\,\,\,\, \young(\oneb2,2)  \,\,\,\,\, \young(\twob1,1) \,\,\,\,\, \young(\twob\oneb,1)  \,\,\,\,\, \young(\twob1,\oneb) \,\,\,\,\, \young(\twob\oneb,\oneb) &  \\
\\
&\young(\twob,1,2) \,\,\,\,\, \young(\twob,\oneb,2)& \\ 
\end{eqnarray*}
By conjecture \ref{con1}, $F_{\omega}^d(\mathbf{x},\mathbf{x})=2s_{3}(\mathbf{x})+6s_{21}(\mathbf{x})+2s_{111}$.\\
  By conjecture \ref{con2}, $F_{\omega}^{d}(\mathbf{x},-\mathbf{x})=0$.\\
By conjecture \ref{con3},
$F_{\omega}^{d}(\mathbf{x},t\mathbf{x})=s_{21}(\mathbf{x})+
ts_{3}(\mathbf{x})+2ts_{21}(\mathbf{x})+ts_{111}(\mathbf{x})+
t^2s_{3}(\mathbf{x})+2t^2s_{21}(\mathbf{x})+t^2s_{111}(\mathbf{x})+
t^3s_{21}(\mathbf{x})$.
 \end{example}

\begin{remark}

We will call a signed Edelman-Greene tableau for $\omega$ a \defn{signed unimodal tableau} for $\omega$ if changing all -$i$ to $i$ for each $i$ produces a \emph{unimodal tableau} for $\omega$ as defined in 3.1 of \cite{Hawkes.Paramonov.Schilling.2017} (after sliding rows to attain a shifted shape).

\begin{example}
The signed unimodal tableaux that appear in example \ref{ex1} are:
\begin{eqnarray*}
&\young(\threeb\oneb01,0) \,\,\,\,\, \young(\oneb013,0) \,\,\,\,\, \young(\threeb\oneb0,\oneb0)  \,\,\,\,\, \young(\threeb01,01)
\end{eqnarray*}
\end{example}

\begin{example}
The signed unimodal tableaux that appear in example \ref{ex2} are:
\begin{eqnarray*}
\young(\twob012) \,\,\,\,\,\,  \young(\twob12,0) \,\,\,\,\, \young(\twob\oneb2,0) 
\end{eqnarray*}
 \end{example}

\begin{example}
The signed unimodal tableaux that appear in example \ref{ex3} are:
\begin{eqnarray*}
\young(\twob12) \,\,\,\,\, \young(\twob\oneb2) \,\,\,\,\, \young(\twob1,1) \,\,\,\,\, \young(\twob\oneb,1)  \,\,\,\,\, \young(\twob1,\oneb) \,\,\,\,\, \young(\twob\oneb,\oneb)  
\end{eqnarray*}
 \end{example}
It is apparent from the definitions of $F_{\omega}^d$ and $F_{\omega}^C$ given in section 1 that $2^{ze(\omega)}F_{\omega}^d(\mathbf{x},\mathbf{x})=F_{\omega}^C(\mathbf{x})$ where $ze(\omega)$ is the number of $s_0$ in a reduced word for $\omega$. From this and equation (3.3) of \cite{Hawkes.Paramonov.Schilling.2017} it can be deduced that:
\begin{eqnarray*}F_{\omega}^d(\mathbf{x},\mathbf{x})=\sum\limits_{\lambda}\bar{U}_{\omega}^{\lambda}P_{\lambda}(\mathbf{x}),
\end{eqnarray*}
where $\bar{U}_{\omega}^{\lambda}$ is the number of signed unimodal tableaux for $\omega$ with shape $\lambda$.  If $\omega$ is unknotted we may apply conjecture \ref{con1} to the left hand side and expand the right hand side in terms of Schur polynomials to get:
\begin{eqnarray*}\sum_{\mu} \bar{E}_{\omega}^{\mu} s_{\mu}(\mathbf{x})=\sum_{\lambda}\sum_{\mu} \bar{U}_{\omega}^{\lambda}h_{\lambda\mu}s_{\mu}(\mathbf{x}),
\end{eqnarray*}
where $h_{\lambda\mu}$ is the multiplicity of $s_{\mu}$ in the Schur expansion of $P_{\lambda}$.
Since $h_{\lambda\lambda}=1$ and $h_{\lambda\mu}=0$ for $\mu \prec \lambda$ in dominance order, it follows that if $\mu$ is maximal such $\bar{E}_{\omega}^{\mu}>0$ then $\bar{E}_{\omega}^{\mu}=\bar{U}_{\omega}^{\mu}$. Hence, for such maximal $\mu$ the set of signed Edelman-Greene tableau for $\omega$ with shape $\mu$ must be equal to the set of signed unimodal tableau for $\omega$ with shape $\mu$.
\end{remark}

\bibliographystyle{alpha}
\bibliography{newstan}

\end{document}